\documentclass[a4paper]{amsart}
\usepackage{graphicx}
\usepackage{amssymb}
\usepackage{amsmath}
\usepackage{amsthm}
\usepackage{amscd}
\usepackage{color}
\usepackage{colordvi}
\usepackage[all,2cell]{xy}
\usepackage{enumitem}
\usepackage[colorlinks=true,citecolor=blue,linkcolor=blue,urlcolor=blue]{hyperref}
\UseAllTwocells \SilentMatrices
\newtheorem{thm}{Theorem}[section]

\newtheorem{cor}[thm]{Corollary}
\newtheorem{lem}[thm]{Lemma}
\newtheorem{exm}[thm]{Example}

\newtheorem{prop}[thm]{Proposition}
\theoremstyle{definition}

\theoremstyle{remark}
\newtheorem{rem}[thm]{\bf Remark}
\numberwithin{equation}{section}
\newcommand{\Mod}{\operatorname{Mod}}
\newcommand{\Hom}{\operatorname{Hom}}
\newcommand{\Ext}{\operatorname{Ext}}
\newcommand{\End}{\operatorname{End}}
\newcommand{\Tor}{\operatorname{Tor}}
\newcommand{\add}{\operatorname{add}}
\newcommand{\Ind}{\operatorname{Ind}}
\newcommand{\Res}{\operatorname{Res}}
\newcommand{\op}{\mathrm{op}}
\newcommand{\KT}{\mathbf{KT}}
\newcommand{\KE}{\mathbf{KE}}
\newcommand{\perpR}[1]{{}^{\perp_R}\!#1}

\newcommand{\pd}{\mathrm{proj.dim}}

\begin{document}
\title[Wakamatsu tilting modules along Frobenius extensions]
{Transfer of Wakamatsu tilting modules along Frobenius extensions}
\author[Wei Ren, Chunxia Zhang] {Wei Ren, Chunxia Zhang$^*$}


\thanks{$^*$Corresponding author; e-mail: cxzhang$\symbol{64}$cqnu.edu.cn}
%\keywords{Wakamatsu tilting module, Frobenius extension}%
%\date{\today}

\begin{abstract}
Let $\iota: R\to A$ be a Frobenius extension and let $T$ be a Wakamatsu tilting left $R$-module. We give sufficient conditions for the induced module $A\otimes_R T$ to remain Wakamatsu tilting and establish an ascent--descent result for split centrally projective Frobenius extensions. We also characterize the tilting case by the vanishing of
$\Ext_R^i(T,A\otimes_R T)$ for all $i>0$. Moreover, if $A\otimes_R T\in\add_R(T)$, then the natural map $S=\End_R(T)\to B=\End_A(A\otimes_R T)$ of endomorphism rings is a Frobenius extension and
$T\otimes_S B\cong A\otimes_R T$ as $R$-$B$-bimodules. Applications to Brenner--Butler--Miyashita equivalences and some specific classes of Frobenius extensions are also discussed.
\end{abstract}

\maketitle
\tableofcontents

\section{Introduction}

\noindent Frobenius extensions form a natural class of ring extensions in which
induction and coinduction coincide. Originating in the work of Kasch \cite{Kas},
Nakayama--Tsuzuku \cite{NT} and Morita \cite{Mor65, Mor67},
they provide an effective setting for transferring homological properties:
if $\iota:R\to A$ is a Frobenius
extension, then the induction $A\otimes_R-:\Mod(R)\rightarrow\Mod(A)$
is both a left and a right adjoint of the restriction. This
two-sided adjunction yields relations between homological
invariants, and has been applied to Gorenstein
homological algebra, tilting theory and related approximation
phenomena; see, for example \cite{BLZ, FXZ, GHZ, LT,  Ren, Xi, Zhao}.

Tilting theory is a natural framework for such transfer
phenomena. Classical tilting modules are closely connected with equivalences
of categories and torsion-theoretic constructions; see, for example
\cite{BB,Hap,HHK,YMiy}. Wakamatsu tilting modules \cite{Wak88, Wak90} extend this
framework by replacing the finite $\add(T)$-coresolutions with proper
infinite ones, making homological transfer more delicate; this is reflected in
the Wakamatsu Tilting Conjecture \cite{BR, DFT}.

We therefore ask when a Wakamatsu tilting $R$-module
$T$ remains Wakamatsu tilting after induction to a Frobenius extension,
i.e. when $U=A\otimes_R T$ is Wakamatsu tilting over $A$.
This question has been considered from
several points of view. For example, Zhang and Wei studied this question under certain conditions
in \cite{ZW}; Bao, L\"u and Zhao proved the preservation for centrally
projective Frobenius extensions \cite{BLZ}, and related results for
generalized tilting modules appear in \cite{FXZ}. 

Our first aim is to isolate the underlying homological condition.
For a Frobenius extension $R\to A$ and a
Wakamatsu tilting $R$-module $T$, Theorem \ref{thm:ascent} shows that $A\otimes_R T$ is
Wakamatsu tilting over $A$ whenever
${}^{\perp_R}T \subseteq {}^{\perp_R}(A\otimes_R T)$,
where ${}^{\perp_R}T$ denotes the left Ext-orthogonal class of $T$.
This condition controls both self-orthogonality and properness of the induced coresolution.
It is automatic if $A\otimes_R T\in\operatorname{add}_R(T)$ as an $R$-module,
so the centrally projective case follows immediately (see Corollary \ref{cor:asc+2}),
recovering the corresponding induction result of \cite{BLZ}.
The Ext-orthogonal formulation makes clear which part of central projectivity is exactly used.

For tilting modules of finite projective dimension, the situation
simplifies: if $T$ is a tilting
$R$-module and $U=A\otimes_R T$, then $U$ is tilting over $A$ if and
only if $\Ext_R^i(T,U)=0$ for all $i>0$ (Proposition \ref{prop:tilting}).
This may be regarded as a module-theoretic analogue for Frobenius extensions of
Miyachi's criterion for tilting complexes \cite{JMiy}. In contrast to the Wakamatsu tilting case,
the finite coresolution is preserved directly by induction, so self-orthogonality is the essential obstruction.

We also study the descent. If $R\to A$ is a split
Frobenius extension, then $T$ is a direct summand of the restricted
$R$-module $A\otimes_R T$. Under the additional condition $A\otimes_R T\in\operatorname{add}_R(T)$,
the Wakamatsu tilting property descends from $A$ to $R$ (Theorem \ref{thm:desc}).
Combining ascent and descent, for a split centrally projective Frobenius extension,
$T$ is Wakamatsu tilting over $R$ if and only if
$A\otimes_R T$ is Wakamatsu tilting over $A$ (Corollary \ref{cor:asc-desc}).

A second theme concerns the induced endomorphism rings. Let
$S=\End_R(T)$, and $B=\End_A(A\otimes_R T)$. We show that if $A\otimes_R T\in\add_R(T)$,
then the natural homomorphism $S\to B$ is a Frobenius extension, and there is a natural isomorphism of
$R$-$B$-bimodules $T\otimes_S B\cong A\otimes_R T$ (Theorem \ref{thm:end-extension}).
No tilting, self-orthogonality, or finite projective dimension assumption is needed.
In particular, for a centrally projective Frobenius extension this holds for every
$R$-module $T$ (Corollary \ref{cor:end-central}), recovering the endomorphism-ring conclusion of
\cite[Theorem A]{BLZ} %for Wakamatsu tilting modules
while removing the additional hypothesis of finite projective dimension required there.
Consequences include Morita-theoretic transfer for progenerators, invariance
of Frobenius extensions under matrix rings, a corner criterion, and a result
for additive generators over Artin algebras; see Corollaries~\ref{cor:morita-end}--\ref{cor:artin}.

There is also a close connection with semidualizing bimodules. If $T$ is a Wakamatsu tilting left
$R$-module and $S=\End_R(T)$, then ${}_RT_S$ is semidualizing by \cite[Corollary~3.2]{Wak04};
the associated Auslander and Bass classes \cite{HW} are naturally available in the present setting.
This viewpoint allows us to relate the two Frobenius extensions $R\to A$ and $S\to B$.
We show compatibility of extension of scalars with the additive equivalences determined by
${}_RT_S$ and ${}_AU_B$ (Proposition \ref{prop:add-compat}), and with the corresponding
Brenner--Butler--Miyashita equivalences when $T$ has finite projective dimension on both sides
(Proposition \ref{prop:BBM}, Corollary \ref{cor:BBM-compat}); compare
\cite{BB, CH, YMiy}.

Finally, we illustrate the results by examples of Frobenius extensions. Finite group ring extensions and truncated
polynomial extensions are split centrally projective, so both ascent and descent apply
(Example \ref{exm:RG}). Skew group rings show that the additive criterion can apply beyond the centrally projective case: although central projectivity may fail, $A\otimes_R T\in\add_R(T)$ holds whenever every twist of
$T$ under the group action belongs to $\add_R(T)$ (Example~\ref{exm:skew-group-end}).

The paper is organized as follows. Section~\ref{sec:Waka} studies ascent and
descent of Wakamatsu tilting modules, including the finite projective dimension case and its relation to Miyachi's theorem. Section \ref{sec:endo} proves the endomorphism ring transfer theorem. Finally, Section \ref{sec:app-exm} develops applications to Brenner--Butler--Miyashita equivalences and gives examples.
\medskip

\section{Ascent and descent of Wakamatsu tilting modules}
\label{sec:Waka}

\noindent Throughout, all rings are associative with a unit, and all modules are left modules unless otherwise specified. We write $\Mod(R)$ for the category of left $R$-modules, and identify right $R$-modules with left $R^{\op}$-modules, where $R^{\op}$ is the opposite ring of $R$. For a module $M$, let $\add(M)$ denote the full subcategory consisting of direct summands of finite direct sums of copies of $M$. For an $R$-module $T$, we denote by
\[ \perpR T := \{X\mid \Ext_R^i(X,T)=0\text{ for all }i>0 \}. \]

An $R$-module $T$ is Wakamatsu tilting \cite{Wak04} if it satisfies the following three conditions: (1) $T$ admits a projective resolution by finitely generated projective modules, (2) $\Ext_R^i(T,T)=0$ for all $i>0$, and (3)
there is an exact sequence
\[ 0\longrightarrow R\longrightarrow T_0\longrightarrow T_1 \longrightarrow T_2\longrightarrow\cdots \]
with each $T_j\in \add_R(T)$, and the sequence remains exact after applying $\Hom_R(-,T)$. Such a sequence is called a proper $\add_R(T)$-coresolution of $R$.

A ring extension $\iota: R\to A$ is a Frobenius extension \cite{Kas, Kad} if
$A$ is finitely generated projective as a left $R$-module and there is an isomorphism
$A\cong \Hom_R(A, R)$ of $A$-$R$-bimodules.
Equivalently, $A$ is finitely generated projective as a right $R$-module and
$A\cong \Hom_{R^{\op}}(A, R)$ as $R$-$A$-bimodules.
Moreover, the Frobenius property of the extension $\iota:R\to A$ is
equivalent to the existence of a natural isomorphism between the
induction and coinduction functors
$ A\otimes_R-\ \cong\ \Hom_R(A,-): \Mod(R) \longrightarrow\Mod(A)$.
In this case, both $(\Ind, \Res)$ and $(\Res, \Ind)$ are adjoint pairs,
where $\Ind = A\otimes_{R}-$, and $\Res: \Mod(A)\rightarrow \Mod(R)$ denotes the restriction functor.

We begin by recording the following standard isomorphisms, which will be
used repeatedly in the sequel. Although they are well known, we include
a brief proof for completeness. Whenever an $A$-module occurs in an $\Ext$-group over $R$, it is
understood to be restricted along $\iota:R\to A$; we suppress $\Res$
from the notation when no ambiguity can arise.

\begin{lem}\label{lem:isoms}%2.1
Let $\iota: R\to A$ be a Frobenius extension. For all $R$-modules $M,N$,
all $A$-modules $X$, and all $i\geq 0$, there are natural isomorphisms
\[ \Ext_A^i(A\otimes_R M, X)\cong \Ext_R^i(M, X), \quad
 \Ext_A^i(X, A\otimes_R N)\cong \Ext_R^i(X, N).\]
Consequently, we have $\Ext_R^i(M, A\otimes_R N)\cong \Ext_R^i(A\otimes_R M,N)$.
\end{lem}

\begin{proof}
Since $A$ is a projective right $R$-module, the induction functor $A\otimes_R-$
is exact, and moreover, it sends projective $R$-modules to projective
$A$-modules. Therefore the usual induction--restriction adjunction
$(\Ind, \Res)$ extends to derived functors, yielding
\[
\Ext_A^i(A\otimes_R M,X)\cong \Ext_R^i(M, X).
\]
Note that any projective resolution of $X$ over $A$ restricts to a projective
resolution of $X$ over $R$. Then, extending the adjunction
$(\Res, \Ind)$ to derived functors, it yields
$\Ext_A^i(X, A\otimes_R N)\cong \Ext_R^i(X, N)$.

Finally, applying the above isomorphisms with $X=A\otimes_RN$ and
$X=A\otimes_RM$ respectively, it gives the isomorphisms
\[\Ext_R^i\bigl(M, A\otimes_RN\bigr)\cong \Ext_A^i(A\otimes_RM,A\otimes_RN)
 \cong \Ext_R^i\bigl(A\otimes_RM, N\bigr).\]
This completes the proof.
\end{proof}

A natural question is whether Wakamatsu tilting modules are preserved under
induction along a Frobenius extension. The derived adjunction shows that the
self-orthogonality of the induced module $U=A\otimes_R T$
is controlled by Ext-groups over $R$. On the other hand, to ensure that a
proper $\add_R(T)$-coresolution of $R$ induces a proper
$\add_A(U)$-coresolution of $A$, one needs the corresponding vanishing for
the successive cosyzygies occurring in the original coresolution. This
motivates the following sufficient condition in terms of the left
Ext-orthogonal class of $T$.

%Let $U=A\otimes_R T$. Since we always suppress $\Res$ from the notation whenever no ambiguity can arise,
%when Ext-groups involving $U$ are taken over $R$, $U$ is understood as the underlying $R$-module.

\begin{thm}\label{thm:ascent}%2.2
Let $\iota: R\to A$ be a Frobenius extension and let $T$ be a Wakamatsu tilting
$R$-module. Assume that $\perpR T\subseteq \perpR U$.
Then $U=A\otimes_R T$ is a Wakamatsu tilting $A$-module.
\end{thm}

\begin{proof}
We verify the defining conditions for $U=A\otimes_R T$ to be a Wakamatsu tilting $A$-module. First, let
$\cdots\rightarrow P_1\rightarrow P_0 \rightarrow T\rightarrow 0$
be a projective resolution of $T$ in which every $P_j$ is finitely generated projective over $R$.
Since $A$ is a projective right $R$-module,
\[ \cdots\longrightarrow A\otimes_R P_1 \longrightarrow A\otimes_R P_0 \longrightarrow U\longrightarrow 0 \]
is a projective resolution of $U$ by finitely generated projective $A$-modules.

We next prove that $U$ is self-orthogonal. Since $T$ is Wakamatsu tilting, it is self-orthogonal over $R$, and hence
$T\in {}^{\perp_R}T$. By the hypothesis ${}^{\perp_R}T\subseteq {}^{\perp_R}U$,
we obtain $\Ext_R^i(T,U)=0$ for all $i>0$. Using the derived induction--restriction adjunction from Lemma~\ref{lem:isoms}, we therefore have the following for all $i>0$:
\[ \Ext_A^i(U,U) = \Ext_A^i(A\otimes_R T,U) \cong \Ext_R^i(T,U) =0. \]
Thus $U$ is self-orthogonal as an $A$-module.

It remains to construct a proper $\add_A(U)$-coresolution of $A$. Since $T$ is Wakamatsu tilting, there exists a proper $\add_R(T)$-coresolution
\begin{equation}\label{eq:proper-T}
0\longrightarrow R\longrightarrow T_0\longrightarrow T_1 \longrightarrow T_2\longrightarrow\cdots,
\end{equation}
where each $T_j\in \add_R(T)$ and \eqref{eq:proper-T} remains exact after applying $\Hom_R(-,T)$. For convenience, put $L_0=R$ and, for each $j\geq 0$, let $L_{j+1}=\mathrm{Coker}(L_j\rightarrow T_j)$. Then \eqref{eq:proper-T} decomposes into short exact sequences
\[ 0\longrightarrow L_j\longrightarrow T_j \longrightarrow L_{j+1}\longrightarrow 0. \]
We claim that $L_j\in{}^{\perp_R}T$ for every $j\geq 0$. Indeed, $L_0=R$ is projective, so the assertion is clear for $j=0$. Since $T_j\in\add_R(T)$ and $T$ is self-orthogonal, $\Ext_R^{>0}(T_j,T)=0$. Furthermore, properness of the exact sequence \eqref{eq:proper-T} implies that applying $\Hom_R(-,T)$ to the above short exact sequence is exact at $\Hom_R(L_j,T)$. Equivalently, $\Ext_R^1(L_{j+1},T)=0$. The long exact $\Ext$-sequence associated to the short exact sequence then gives, for every $n\geq 1$,
$\Ext_R^{n+1}(L_{j+1},T) \cong \Ext_R^n(L_j,T)$. An induction on $j$ therefore yields $\Ext_R^i(L_j,T)=0$ for all $i>0$, as claimed.

Hence, by the assumed inclusion ${}^{\perp_R}T\subseteq{}^{\perp_R}U$, it follows that
$\Ext_R^i(L_j,U)=0$ for every $j\geq0$ and every $i>0$. Now apply the exact functor $A\otimes_R-$ to the proper $\add_R(T)$-coresolution \eqref{eq:proper-T}. Since $A\otimes_RR\cong A$, we obtain an exact sequence
\begin{equation}\label{eq:induced-coresolution}
0\longrightarrow A\longrightarrow A\otimes_R T_0 \longrightarrow A\otimes_R T_1 \longrightarrow A\otimes_R T_2 \longrightarrow\cdots.
\end{equation}
Because $T_j\in\add_R(T)$, we have $A\otimes_R T_j\in \add_A(A\otimes_R T) = \add_A(U)$
for every $j$. It only remains to verify that \eqref{eq:induced-coresolution} is $\Hom_A(-,U)$-exact.
Applying $A\otimes_R-$ to the short exact sequence $0\rightarrow L_j\rightarrow T_j \rightarrow L_{j+1}\rightarrow 0$,
we obtain a short exact sequence
\[ 0\longrightarrow A\otimes_R L_j \longrightarrow A\otimes_R T_j \longrightarrow A\otimes_R L_{j+1} \longrightarrow0. \]
Invoking Lemma~\ref{lem:isoms}, we have
$\Ext_A^1(A\otimes_R L_{j+1},U) \cong \Ext_R^1(L_{j+1},U) =0$.  Consequently, applying $\Hom_A(-,U)$ to the preceding short exact sequence yields
\[ 0\rightarrow \Hom_A(A\otimes_R L_{j+1},U) \rightarrow \Hom_A(A\otimes_R T_j,U) \rightarrow \Hom_A(A\otimes_R L_j,U) \rightarrow 0. \]
Thus \eqref{eq:induced-coresolution} remains exact under $\Hom_A(-,U)$. Hence it is a proper $\add_A(U)$-coresolution of $A$. We have shown that $U$ has a projective resolution by finitely generated projective $A$-modules, is self-orthogonal, and admits a proper $\add_A(U)$-coresolution of $A$. Therefore $U$ is a Wakamatsu tilting $A$-module.
\end{proof}

\begin{cor}\label{cor:asc+1}%2.3
Let $\iota: R\to A$ be a Frobenius extension and let $T$ be a Wakamatsu tilting
$R$-module. If $A\otimes_R T\in\add_R(T)$ as an $R$-module,
then $A\otimes_R T$ is Wakamatsu tilting over $A$.
\end{cor}

\begin{proof}
Let $M$ be an $R$-module in $\perpR T$, that is, $\Ext^i_R(M, T)=0$ for all $i>0$. By the assumption
$U = A\otimes_R T\in\add_R(T)$, it follows readily that $\Ext_R^i(M, U)=0$ for all $i>0$.
Hence, $\perpR T\subseteq \perpR U$, and the assertion follows from the above theorem.
\end{proof}

Following Hirata \cite{Hira}, an $R$-$R$-bimodule $M$ is called centrally projective over $R$ if
${}_R M_R\in\add_{R\text{-}R}({}_R R_R)$, that is, if $M$ is isomorphic, as an $R$-$R$-bimodule, to a direct summand of $R^n$ for some $n\geq 1$. A Frobenius extension $\iota: R\to A$ is called centrally projective if
${}_R A_R$ is centrally projective over $R$. If ${}_R A_R$ is centrally projective, then
$A\otimes_R T \in\add_R(T)$ for every $R$-module $T$.

\begin{cor}\label{cor:asc+2}%2.4
Let $\iota: R\to A$ be a centrally projective Frobenius extension. Then for
every Wakamatsu tilting $R$-module $T$, the induced $A$-module
$A\otimes_R T$ is Wakamatsu tilting.
\end{cor}

\begin{rem}\label{rem:asc-BLZ}%2.5
Corollary~\ref{cor:asc+2} agrees with \cite[Proposition~3.2]{BLZ}. A related induction theorem for
generalized tilting modules over a commutative base ring appears in
\cite{FXZ}. The above result emphasizes that the inclusion $\perpR T\subseteq\perpR(A\otimes_R T)$ is essential; central
projectivity is one convenient sufficient condition.
\end{rem}

We refer to \cite{YMiy} for the notion of tilting modules (of finite projective dimension).
Every tilting module is Wakamatsu tilting, whereas a Wakamatsu tilting
module need not be tilting. A longstanding open problem is the
Wakamatsu Tilting Conjecture; see \cite[WT-Conjecture, p.~71]{BR}. It
asserts that a Wakamatsu tilting module of finite projective dimension
is necessarily a tilting module.

In the setting of finite projective dimension, the transfer problem along
a Frobenius extension admits a particularly simple characterization.
Indeed, the finite $\add(T)$-coresolution required in the definition of
a tilting module is automatically preserved by induction, so that the
only additional condition to be checked is the self-orthogonality of
the induced module.

\begin{prop}\label{prop:tilting}%2.6
Let $\iota: R\to A$ be a Frobenius extension, and let $T$ be a tilting
$R$-module of finite projective dimension. For the induced $A$-module
$U=A\otimes_R T$, it is tilting if and only if
$\Ext_R^i(T, U)=0$ for every $i>0$.
\end{prop}

\begin{proof}
Suppose first that $U$ is a tilting $A$-module. Then $U$ is
self-orthogonal, and the derived induction--restriction adjunction of
Lemma~\ref{lem:isoms} yields
\[\Ext_R^i(T,U)\cong\Ext_A^i(A\otimes_R T,U)=\Ext_A^i(U,U)=0 \]
for all $i>0$.

Conversely, assume that $\Ext_R^i(T,U)=0$ for any $i>0$.
We verify the three defining conditions for $U$ to be a tilting
$A$-module. Since $T$ has finite projective dimension over $R$, by applying $A\otimes_R-$
to a finite projective resolution of $T$, we then obtain a finite projective resolution of $U$.
Hence, $\pd_A U\leq \pd_R T<\infty$.
Next, by Lemma~\ref{lem:isoms}, we have
$\Ext_A^i(U,U)=\Ext_A^i(A\otimes_R T,U)\cong\Ext_R^i(T,U)=0$
for all $i>0$, so $U$ is self-orthogonal.
Finally, since $T$ is a tilting $R$-module, there exists an exact sequence
\[0\longrightarrow R\longrightarrow T_0\longrightarrow \cdots\longrightarrow
T_n\longrightarrow 0,\]
where $T_i\in \add_R(T)$.
Analogous to the above, applying $A\otimes_R-$ to this exact sequence,
we can show that $A$ admits a finite $\add_A(U)$-coresolution.
Hence, $U$ is a tilting $A$-module.
\end{proof}

\begin{rem}\label{rem:miyachi}%2.7
The preceding result is closely related to Miyachi's extension theorem for
tilting complexes \cite{JMiy}. Since $A$ is a projective right $R$-module,
$A\otimes_R^{\mathbf L}T\cong A\otimes_R T$.
Thus, for a tilting $R$-module $T$ of finite projective dimension,
Miyachi's condition $\Hom_{\mathbf D^b(R)}\bigl(T,A\otimes_R T[i]\bigr)=0$
for all $i\neq0$ is equivalent to $\Ext_R^i(T,A\otimes_R T)=0$ for all $i>0$,
since the negative-degree Hom-groups vanish for modules concentrated in
degree zero. Thus Proposition~\ref{prop:tilting} may be viewed as the module-theoretic
specialization of Miyachi's criterion to Frobenius extensions, together
with the converse supplied by the derived Frobenius adjunction.
\end{rem}

It is natural to ask whether Wakamatsu tilting modules also descend along
Frobenius extensions. In contrast with the ascent problem, we need splitness of the
extension, which allows one to recover $T$ as a direct summand of the restricted
induced module. Recall that an extension $\iota: R\to A$ is called split if the
inclusion admits an $R$-$R$-bimodule retraction; equivalently,
${}_RA_R\cong{}_RR_R\oplus{}_RC_R$ for some $R$-$R$-bimodule $C$.
%However, splitness alone does not guarantee that a proper $\add_A(A\otimes_R T)$-coresolution descends to a proper $\add_R(T)$-coresolution after restriction of scalars. We therefore impose a mild additional condition ensuring the required compatibility.

\begin{thm}\label{thm:desc}%2.8
Let $\iota:R\to A$ be a split Frobenius extension and let $T$ be an
$R$-module. Assume that $U=A\otimes_R T\in\add_R(T)$ as an $R$-module.
If $U$ is a Wakamatsu tilting $A$-module, then $T$ is a Wakamatsu tilting $R$-module.
\end{thm}

\begin{proof}
Since the extension is split, there is an $R$-$R$-bimodule $C$ such that
${}_R A_R\cong {}_R R_R\oplus {}_R C_R$.
Applying $-\otimes_R T$ to this decomposition gives an isomorphism of
$R$-modules
$U = A\otimes_R T \cong T\oplus(C\otimes_R T)$; in particular, this implies
$T\in\add_R(U)$. Together with the assumption $U\in\add_R(T)$, it follows that
$\add_R(U)=\add_R(T)$.

We first verify the finiteness condition in the definition of a Wakamatsu
tilting module. Since $U$ is Wakamatsu tilting over $A$, it admits a
projective resolution
\[
\cdots\longrightarrow P_1\longrightarrow P_0
\longrightarrow U\longrightarrow 0
\]
in which every $P_i$ is finitely generated projective over $A$. Since
$A$ is finitely generated projective $R$-module, restriction of scalars sends
finitely generated projective $A$-modules to finitely generated projective
$R$-modules. Hence, as a restricted $R$-module, $U$ admits a projective resolution by finitely
generated projective $R$-modules. Since the class of modules admitting such a
resolution is closed under direct summands, and $U = A\otimes_R T \cong T\oplus(C\otimes_R T)$
as $R$-modules, $T$ also admits a projective resolution by finitely
generated projective $R$-modules.

We next prove that $T$ is self-orthogonal. By the derived
induction--restriction adjunction, we have
$\Ext_R^i(T,U)\cong \Ext_A^i(A\otimes_R T,U)=\Ext_A^i(U,U)=0$
for all $i>0$. Since $T$ is a direct summand of ${}_R U$, the group
$\Ext_R^i(T,T)$ is a direct summand of $\Ext_R^i(T,U)$. Therefore
$\Ext_R^i(T,T)=0$ for all $i>0$.

It remains to construct a proper $\add_R(T)$-coresolution of $R$.
Since $U$ is Wakamatsu tilting over $A$, there exists a proper
$\add_A(U)$-coresolution
\begin{equation}\label{eq:proper-A}
0\longrightarrow A\longrightarrow U_0\longrightarrow U_1
\longrightarrow U_2\longrightarrow\cdots,
\end{equation}
where $U_j\in\add_A(U)$ and it remains exact after applying $\Hom_A(-,U)$.
Restricting scalars along $\iota$ preserves exactness, and then by $U_j\in\add_A(U)$
we imply that $U_j\in\add_R(U)=\add_R(T)$ as restricted $R$-modules.

We claim that,  as a restricted exact sequence of $R$-modules, (\ref{eq:proper-A})
remains exact after applying $\Hom_R(-, T)$. Indeed, since $\iota: R\to A$ is a Frobenius extension,
$U=A\otimes_R T \cong \Hom_R(A,T)$ as left $A$-modules, and therefore, for every $A$-module $X$ we have
natural isomorphisms
\[\Hom_A(X,U)\cong \Hom_A\bigl(X,\Hom_R(A,T)\bigr)\cong\Hom_R(X, T).\]
Then, applying $\Hom_A(-,U)$ to the above proper $\add_A(U)$-coresolution  (\ref{eq:proper-A}), we
have an exact sequence
\[ \cdots \longrightarrow \Hom_R(U_1, T)\longrightarrow\Hom_R(U_0, T)\longrightarrow\Hom_R(A, T)\longrightarrow 0,\]
which implies that the restriction of (\ref{eq:proper-A}) is a proper $\add_R(T)$-coresolution of the $R$-module $A$, as needed.

Since $T$ is self-orthogonal, it follows from \cite[Lemma~2.2]{Wak04} that the class of modules admitting a proper
$\add_R(T)$-coresolution is closed under direct summands. Hence, as a direct
summand of the restricted $R$-module $A$, $R$ admits a proper $\add_R(T)$-coresolution.
Therefore $T$ is a Wakamatsu tilting $R$-module.
\end{proof}

As an immediate consequence of Corollary~\ref{cor:asc+2} and
Theorem~\ref{thm:desc}, we obtain the following ascent and descent
result.

\begin{cor}\label{cor:asc-desc}%2.9
Let $\iota: R\to A$ be a split, centrally projective Frobenius extension,
and let $T$ be an $R$-module. Then
$T$ is Wakamatsu tilting over $R$ if and only if
$A\otimes_R T$ is Wakamatsu tilting over $A$.
\end{cor}

\begin{proof}
It suffice to prove the ``if'' part. Suppose that $A\otimes_R T$ is Wakamatsu tilting over $A$.
Since ${}_R A_R$ is centrally projective,
${}_R A_R\in\add({}_R R_R)$. Then, it follows readily that $A\otimes_R T\in\add_R(T)$.
Moreover, since the Frobenius extension $\iota: R\to A$ is split,
Theorem~\ref{thm:desc} applies and shows that $T$ is Wakamatsu
tilting over $R$.
\end{proof}
\medskip

\section{Endomorphism rings}
\label{sec:endo}

\noindent We now turn to the behavior of endomorphism rings under Frobenius
extensions. Let $\iota: R\to A$ be a ring extension and let $T$ be an $R$-module.
Throughout this section, set $U=A\otimes_R T$, $S=\End_R(T)$ and $B=\End_A(U)$.
Endomorphism rings are written according to
the right-action convention: for $s_1,s_2\in S=\End_R(T)$,
$t(s_1s_2)=(ts_1)s_2$, or equivalently, $s_1s_2=s_2\circ s_1$
in the usual notation for composition of maps $T\stackrel{s_1}\rightarrow T \stackrel{s_2}\rightarrow T$. Thus $T$ is naturally
an $R$-$S$-bimodule via $ts=s(t)$. Analogously, we regard $U$ as an $A$-$B$-bimodule.
The assignment $\rho: S\rightarrow B$ given by $\rho(s)=1_A\otimes s$
is a ring homomorphism; explicitly,
$(a\otimes t)\rho(s)=a\otimes ts$ for any
$a\in A$, $t\in T$ and $s\in S$.

We first record the corresponding evaluation and duality isomorphisms, which are
immediate for $M=T^n$.  The two constructions are compatible with finite direct sums and direct summands.
Hence, such isomorphisms hold for any $R$-module $M\in\add_R(T)$.

\begin{lem}\label{lem:yoneda}%3.1
Let $M\in\add_R(T)$. Then the following hold:
\begin{enumerate}
\item[(i)] the evaluation map
$\varepsilon_M: T\otimes_S\Hom_R(T,M)\rightarrow M$ given by
$t\otimes f\mapsto f(t)$ is an isomorphism;

\item[(ii)] the natural map $\delta_M: \Hom_R(M,T)\rightarrow \Hom_S(\Hom_R(T,M),S)$
defined by $\delta_M(g)(f)=g\circ f$ is an isomorphism.
\end{enumerate}
Both isomorphisms are natural in $M$ and are compatible with additional
module structures induced by endomorphisms of $M$.
\end{lem}

When $T$ is Wakamatsu tilting, the bimodule ${}_RT_S$ also reflects the
usual two-sided symmetry of Wakamatsu tilting modules. In particular,
$T_S$ is a Wakamatsu tilting right $S$-module and the canonical map
$R\rightarrow \End_{S^{\op}}(T)$ is an isomorphism; see \cite[Corollary~3.2]{Wak04}.

Classical endomorphism-ring results for Frobenius extensions were developed
by Morita; see \cite{Mor65, Mor67}. Motivated by this theory, we establish
the following additive criterion adapted to the present setting.
Its hypothesis involves only the additive relation between $T$ and its induced module $U=A\otimes_RT$;
no assumption of tilting, self-orthogonality, or finiteness of projective dimension
is required.

\begin{thm}\label{thm:end-extension}%3.2
Let $\iota: R\to A$ be a Frobenius extension and let $T$ be an
$R$-module. Assume that $U\in\add_R(T)$ as an $R$-module. Then
the natural homomorphism $\rho:S\to B$ is a Frobenius extension,
and moreover,  the canonical evaluation map
$T\otimes_S B\to U$ defined by
$t\otimes b\mapsto (1\otimes t)b$
is an isomorphism of $R$-$B$-bimodules.
\end{thm}

\begin{proof}
By the adjunction $\Hom_A(A\otimes_R T,U)\cong\Hom_R(T,U)$, it gives a natural isomorphism
of $S$-$B$-bimodules $B\cong \Hom_R(T,U)$. %{eq:B-as-hom}
The left $S$-action on $\Hom_R(T,U)$ is given by
$(sf)(t)=f(ts)$, whereas the right $B$-action is given by
$(fb)(t)=f(t)b$.
By the assumption $U\in\add_R(T)$, we have
$\Hom_R(T,U)\in S\text{-}\mathrm{proj}$,
where $S\text{-}\mathrm{proj}$ denotes the category of finitely generated
projective $S$-modules. Therefore $B$ is a finitely generated projective $S$-module.
%There is a standard equivalence $\Hom_R(T,-): \add_R(T)\xrightarrow{\ \sim\ }S\text{-}\mathrm{proj}$, with quasi-inverse $T\otimes_S-$.

Since $\iota: R\to A$ is a Frobenius extension, we obtain
\[ B =\Hom_A(U, A\otimes_R T)\cong \Hom_A(U, \Hom_R(A, T))\cong \Hom_R(U,T). \]
Taking the natural module structures into account, this is an
isomorphism of $B$-$S$-bimodules $B\cong \Hom_R(U,T)$. %{eq:B-other-hom}
Since $U\in\add_R(T)$, Lemma~\ref{lem:yoneda}(ii) gives an isomorphism
\[ \Hom_R(U,T)\cong \Hom_S(\Hom_R(T,U),S).\]
Moreover, this is an isomorphism of $B$-$S$-bimodules, where the
left $B$-action on the right-hand side is induced from the right
$B$-action on $\Hom_R(T,U)$.
Combining with %{eq:B-as-hom} and {eq:B-other-hom}
the isomorphisms $B\cong \Hom_R(T,U)$ of $S$-$B$-bimodules, and $B\cong \Hom_R(U,T)$
of $B$-$S$-bimodules, we obtain an isomorphism of $B$-$S$-bimodules
$B\cong \Hom_S(B, S)$. Since $B$ is finitely generated projective as a left $S$-module, this is precisely the
Frobenius condition for the extension $\rho:S\to B$.

Using the isomorphism $B\cong \Hom_R(T,U)$ %{eq:B-as-hom}
and the evaluation isomorphism of Lemma~\ref{lem:yoneda}(i), we obtain
\[
T\otimes_S B
\cong
T\otimes_S\Hom_R(T,U)
\xrightarrow{\ \sim\ }U.
\]
Under the adjunction identification
$B\cong\Hom_R(T,U)$, this map is explicitly given by
$t\otimes b\mapsto(1\otimes t)b$.
It is clearly compatible with the left $R$-action and the right
$B$-action. Hence it is an isomorphism of $R$-$B$-bimodules.
\end{proof}

The second assertion above admits a coinduction interpretation.

\begin{cor}\label{cor:coind}%3.3
Under the hypotheses of Theorem~\ref{thm:end-extension}, there is a
natural isomorphism of $R$-$B$-bimodules
$U \cong \Hom_{S^{\op}}(B, T)$.
\end{cor}

\begin{proof}
Since $\rho: S\to B$ is a Frobenius extension, for the right $S$-module $T$, we have
$T\otimes_S B \cong \Hom_{S^{\op}}(B, T)$.
The left $R$- and right $B$-actions are preserved by naturality.
By Theorem~\ref{thm:end-extension}, $U \cong T\otimes_S B$. Hence, the required isomorphism
$U \cong \Hom_{S^{\op}}(B, T)$ holds.
\end{proof}

If $\iota: R\to A$ is centrally projective, then ${}_R A_R\in\add_{R\text{-}R}({}_R R_R)$,
and hence $A\otimes_R T\in\add_R(T)$ for every $R$-module $T$. Therefore
Theorem~\ref{thm:end-extension} applies to every $T$.
In particular, for a Wakamatsu tilting module $T$,
this recovers the endomorphism ring conclusion of \cite[Theorem~3.5]{BLZ},
while showing that neither the Wakamatsu tilting hypothesis nor the additional assumption of finite
projective dimension is needed for the endomorphism ring extension itself.

\begin{cor}\label{cor:end-central}%3.4
Let $\iota: R\to A$ be a centrally projective Frobenius extension.
Then, for every $R$-module $T$, the natural homomorphism of endomorphism rings
$\rho: S\to B$
is a Frobenius extension. Moreover, there is an isomorphism of $R$-$B$-bimodules
$T\otimes_S B\cong A\otimes_R T$.
\end{cor}

The following result may be viewed as the additive
Morita-theoretic form of the invariance of Frobenius extensions.
Miyachi's results in \cite{JMiy} provide a derived analogue in which
progenerators are replaced by suitable tilting complexes.

\begin{cor}\label{cor:morita-end}%3.5
Let $\iota: R\to A$ be a Frobenius extension and let $T$ be a finitely
generated projective generator of $\Mod(R)$.
Then $U = A\otimes_R T$ is a finitely generated projective generator of
$\Mod(A)$, and the natural homomorphism of endomorphism rings $\rho: S\to B$ is a Frobenius extension.
\end{cor}

\begin{proof}
Since $T$ is a finitely generated projective generator,
$\add_R(T)$ is precisely $R\text{-}\mathrm{proj}$, the class of finitely generated projective $R$-modules.
The module $T$ is a direct summand of $R^n$ for some $n$, and hence
$U= A\otimes_R T$ is a direct summand of $A^n$ as $R$-modules. Since $A$ is finitely generated
projective as an $R$-module, $U\in\add_R(T)$.
Theorem~\ref{thm:end-extension} now yields the Frobenius extension
$\rho: S\to B$.

Moreover, $A\otimes_R T$ is finitely generated projective over $A$.
Since $R\in\add_R(T)$, applying $A\otimes_R-$ gives
$A\in\add_A(U)$, so $U$ is an $A$-progenerator.
\end{proof}

Taking the free progenerator gives the following familiar consequence
for matrix rings.

\begin{cor}\label{cor:matrix}%3.6
Let $R\to A$ be a Frobenius extension and let $m\geq1$. Then
$M_m(R)\longrightarrow M_m(A)$ is a Frobenius extension.
\end{cor}

\begin{proof}
Take $T=R^m$. With the convention of endomorphism ring,
$\End_R(R^m)\cong M_m(R)$ and $\End_A(A^m)\cong M_m(A)$.
The assertion follows directly from Corollary~\ref{cor:morita-end}.
\end{proof}

The same argument yields a useful corner construction.

\begin{cor}\label{cor:corner}%3.7
Let $\iota: R\to A$ be a Frobenius extension and let $e\in R$ be an idempotent.
If $Ae\in\add_R(Re)$ as a restricted $R$-module, then the natural homomorphism
$eRe\rightarrow eAe$ is a Frobenius extension.
In particular, the conclusion holds for every idempotent $e\in R$ if
$\iota: R\to A$ is centrally projective.
\end{cor}

\begin{proof}
Take $T=Re$. Then $A\otimes_RRe\cong Ae$. Note that
$\End_R(Re)\cong eRe$ and $\End_A(Ae)\cong eAe$.
Thus the first assertion is an immediate consequence of
Theorem~\ref{thm:end-extension}.

If $\iota: R\to A$ is centrally projective, then
${}_RA_R\in\add_{R\text{-}R}({}_RR_R)$.
Tensoring with $Re$ gives $Ae\in\add_R(Re)$ as an $R$-module, and the second assertion follows.
\end{proof}

For Artin algebras, the additive hypothesis is automatic when one starts
with an additive generator. Recall that an Artin algebra $R$ is of finite
representation type if and only if the category of finitely generated $R$-modules
$\mathrm{mod}(R)$ admits an additive generator. If $R$ is of finite representation type and
$T$ is a basic additive generator of $\mathrm{mod}(R)$, that is,
$T=\bigoplus_{i=1}^n T_i$, where $T_1,\ldots,T_n$ form a complete set of representatives of the
isomorphism classes of indecomposable finitely generated $R$-modules,
then the corresponding endomorphism algebra is the Auslander algebra
associated with $R$; see \cite[Chapter~VI, Section~5]{ARS}.

Consequently, the following corollary yields a Frobenius extension from
the Auslander endomorphism algebra associated with $R$ to the
endomorphism algebra of the induced module $A\otimes_R T$. If
$A\otimes_R T$ is also an additive generator of $\mathrm{mod}(A)$,
then $A$ is of finite representation type and the target endomorphism
algebra is likewise the Auslander algebra associated with $A$.
It is worth noting that we use our fixed right-action convention for
endomorphism rings; this differs by an opposite ring from the convention
used in \cite[Chapter~VI, Section~5]{ARS}.

\begin{cor}\label{cor:artin}%3.8
Let $\iota: R\to A$ be a Frobenius extension of Artin algebras and let $T$ be
an additive generator of $\mathrm{mod}(R)$. Then the natural homomorphism of endomorphism rings
$\rho: S\to B$ is a Frobenius extension.
\end{cor}

\begin{proof}
As $A$ is a finitely generated projective $R$-module, it yields that $A\otimes_R T$ is also a
finitely generated $R$-module, that is, $A\otimes_R T\in\mathrm{mod}(R)$.
Since $T$ is an additive generator of $\mathrm{mod}(R)$, we have
$\mathrm{mod}(R)=\add_R(T)$, and therefore
$A\otimes_R T\in\add_R(T)$. The assertion now follows from Theorem~\ref{thm:end-extension}.
\end{proof}

\medskip

\section{Applications and examples}
\label{sec:app-exm}

\noindent We now record several consequences of the preceding results. Throughout this section, we
assume, unless otherwise stated, that $\iota:R\to A$ is a Frobenius extension, $T$ is an $R$-module,
$U=A\otimes_RT$, $S=\End_R(T)$, $B=\End_A(U)$, and $U\in\add_R(T)$ as an $R$-module.
By Theorem~\ref{thm:end-extension}, the natural homomorphism $\rho:S\to B$ is then a Frobenius
extension and there is an isomorphism of $R$-$B$-bimodules $U\cong T\otimes_SB$.

We first show that the Frobenius extensions $\iota: R\to A$ and $\rho: S\to B$
are compatible with the equivalences determined by
${}_RT_S$ and ${}_AU_B$. Consider the standard mutually quasi-inverse equivalences
and the following diagram
\[
\xymatrix@C=80pt@R=40pt{
S\text{-}\mathrm{proj}
 \ar@<0.5ex>[r]^{T\otimes_S-}
 \ar@<-1ex>[d]_{B\otimes_S-}
&\add_R(T)
 \ar@<0.5ex>[l]^{\Hom_R(T,-)}
 \ar@<1ex>[d]^{A\otimes_R-}\\
B\text{-}\mathrm{proj}
 \ar@<0.5ex>[r]^{U\otimes_B-}
 \ar@<-1ex>[u]_{\Res}
&\add_A(U)
 \ar@<0.5ex>[l]^{\Hom_A(U,-)}
 \ar@<1ex>[u]^{\Res}.
}
\]

%\[T\otimes_S-: S\text{-}\mathrm{proj}\rightleftarrows \add_R(T): \Hom_R(T,-), U\otimes_B-: B\text{-}\mathrm{proj} \rightleftarrows \add_A(U): \Hom_A(U,-).\]

\begin{prop}\label{prop:add-compat}%4.1
Under the above assumptions, extension of scalars is compatible with
these equivalences. More precisely, the following
natural isomorphisms hold:

\begin{enumerate}
\item[(i)] for every $P\in S\text{-}\mathrm{proj}$,
$A\otimes_R(T\otimes_S P)\cong U\otimes_B(B\otimes_S P)$;

\item[(ii)] for every $X\in \add_A(U)$,
$\Hom_R(T, \Res X)\cong \Res\Hom_A(U,X)$
as left $S$-modules, where the restriction on the right-hand side is
along $\rho:S\to B$;

\item[(iii)] for every $M\in\add_R(T)$,
$B\otimes_S\Hom_R(T,M)\cong \Hom_A(U,A\otimes_R M)$
as left $B$-modules;

\item[(iv)] for every $Q\in B\text{-}\mathrm{proj}$,
$T\otimes_S \Res Q\cong \Res(U\otimes_B Q)$.
\end{enumerate}
\end{prop}

\begin{proof}
The first isomorphism follows immediately from associativity of tensor products:
\[A\otimes_R(T\otimes_SP)\cong U\otimes_SP\cong U\otimes_B(B\otimes_SP).\]
The second isomorphism is the induction--restriction adjunction.

For $M\in\add_R(T)$, both $B\otimes_S\Hom_R(T,M)$ and
$\Hom_A(U,A\otimes_RM)$ are finitely generated projective left $B$-modules. Applying
$U\otimes_B-$ to the first one gives
\[
\begin{aligned}
U\otimes_B
 \bigl(B\otimes_S\Hom_R(T,M)\bigr)
&\cong
U\otimes_S\Hom_R(T,M)\\
&\cong
A\otimes_R
 \bigl(T\otimes_S\Hom_R(T,M)\bigr)\\
&\cong
A\otimes_RM.
\end{aligned}
\]
For $M\in\add_R(T)$, it is clear that $A\otimes_R M\in \add_A(U)$. Hence, the standard
mutually quasi-inverse equivalence shows that $U\otimes_B\Hom_A(U,A\otimes_RM)\cong A\otimes_RM$.
Since $U\otimes_B-$ is fully faithful on $B\text{-}\mathrm{proj}$, the third isomorphism follows.

Recall that there is an isomorphism of $R$-$B$-bimodules $U\cong T\otimes_SB$. Hence,
for every $Q\in B\text{-}\mathrm{proj}$, the forth isomorphism holds.
\end{proof}

The preceding proposition concerns only additive closures. For the corresponding Ext- and Tor-orthogonal classes, the appropriate general setting is given by the Auslander and Bass classes. We refer to \cite[Definition 2.1]{HW}
the notion of semidualizing modules over associative rings. Let ${}_RT_S$ be a Wakamatsu tilting bimodule.
It follows from \cite[Corollary 3.2]{Wak04} that ${}_RT_S$ is semidualizing.
The associated Auslander and Bass classes are as follows:
\[ \mathcal{A}_T(S) = \left\{ M\in \Mod(S) \, \vline \,
\begin{matrix} \Tor_i^S(T, M)=\Ext_R^i(T,T\otimes_S M)=0, \forall i>0,
\\
 \,M\rightarrow\Hom_R(T,T\otimes_S M) \text{ is an isomorphism}
 \end{matrix} \right\},
\]
and
\[ \mathcal{B}_T(R) = \left\{ N\in \Mod(R) \, \vline \,
\begin{matrix} \Ext_R^i(T, N)=\Tor_i^S(T,\Hom_R(T,N))=0, \forall i>0,
\\
 \,T\otimes_S\Hom_R(T,N)\longrightarrow N \text{ is an isomorphism}
 \end{matrix} \right\}.
\]
%Define $\mathcal A_T(S)$ to consist of all $N\in S\text{-}\Mod$ such that $\Tor_i^S(T,N)=\Ext_R^i(T,T\otimes_SN)=0$ for all $i>0$, and the canonical unit $N\rightarrow\Hom_R(T,T\otimes_SN)$ is an isomorphism. Dually, let $\mathcal B_T(R)$ consist of all $M\in R\text{-}\Mod$ such that for all $i>0$, $\Ext_R^i(T,M)=\Tor_i^S(T,\Hom_R(T,M))=0$, and the evaluation map $T\otimes_S\Hom_R(T,M)\longrightarrow M$ is an isomorphism.
The standard Foxby equivalence (see, for example \cite[Proposition 4.1]{HW}) gives
\[ \xymatrix@C=80pt{
\mathcal A_T(S) \ar@<0.6ex>[r]^{T\otimes_S-}&\mathcal B_T(R)
 \ar@<0.6ex>[l]^{\Hom_R(T,-)} .} \]

We consider the subclasses
\[\KT(T_S)= \{\,M\in \Mod(S)\mid \Tor_i^S(T, M)=0\text{ for all }i>0\,\},\]
\[ \KE({}_RT)=\{\,N\in \Mod(R) \mid \Ext_R^i(T, N)=0\text{ for all }i>0\,\},\]
and have the following Brenner--Butler--Miyashita equivalence; compare
\cite[Section III.3]{Hap} and \cite[Theorem 1.16]{YMiy}.

\begin{prop}\label{prop:BBM}%4.2
Let ${}_RT_S$ be a Wakamatsu tilting bimodule. If $T$ has finite projective
dimension on both sides, then the adjunction %{eq:BBM}
\[T\otimes_S-:\KT(T_S)\rightleftarrows\KE({}_RT):\Hom_R(T,-) \]
is an equivalence of categories. Moreover, $\mathcal A_T(S)=\KT(T_S)$ and $\mathcal B_T(R)=\KE({}_RT)$.
\end{prop}

\begin{proof}
By \cite[Corollary~4.2]{CH}, the adjoint pair
\[T\otimes_S-:\Mod(S)\rightleftarrows\Mod(R):\Hom_R(T,-)\]
restricts to an equivalence
$\KT(T_S)\simeq\KE({}_RT)$.
Thus, for every $M\in\KT(T_S)$, one has
$T\otimes_S M\in\KE({}_RT)$,
and the adjunction unit
$M\rightarrow\Hom_R(T,T\otimes_S M)$
is an isomorphism. Hence
$\Ext_R^i(T,T\otimes_S M)=0$ for all $i>0$.
So, the additional conditions defining the Auslander class are
automatically satisfied. Therefore $\mathcal A_T(S)=\KT(T_S)$.

Dually, for every $N\in\KE({}_RT)$, one has
$\Hom_R(T, N)\in\KT(T_S)$, and the adjunction counit
$T\otimes_S\Hom_R(T, N)\rightarrow N$
is an isomorphism. In particular,
$\Tor_i^S\bigl(T,\Hom_R(T, N)\bigr)=0$ for all $i>0$.
Consequently, $\mathcal B_T(R)=\KE({}_RT)$.
\end{proof}

We now combine this Brenner--Butler--Miyashita equivalence with the preceding results on
Frobenius extensions.

\begin{cor}\label{cor:BBM-compat}%4.3
Let $\iota: R\to A$ be a Frobenius extension and let $T$ be a
Wakamatsu tilting $R$-module. Assume that
$U =A\otimes_RT \in\add_R(T)$ and that $T$ has finite
projective dimension both as a left $R$-module and as a right
$S$-module.
Then $U$ is Wakamatsu tilting over $A$, $\rho: S\to B$ is a Frobenius
extension, and $U$ has finite projective dimension both
as a left $A$-module and as a right $B$-module.

Moreover, the Brenner--Butler--Miyashita equivalences associated
with ${}_RT_S$ and ${}_AU_B$ are compatible with extension and
restriction of scalars. More precisely, the following diagram commutes
up to natural isomorphism:
\[
\xymatrix@C=80pt@R=40pt{
\KT(T_S)
 \ar@<0.5ex>[r]^{T\otimes_S-}
 \ar@<-1ex>[d]_{B\otimes_S-}
&
\KE({}_RT)
 \ar@<0.5ex>[l]^{\Hom_R(T,-)}
 \ar@<1ex>[d]^{A\otimes_R-}
\\
\KT(U_B)
 \ar@<0.5ex>[r]^{U\otimes_B-}
 \ar@<-1ex>[u]_{\Res}
&
\KE({}_AU)
 \ar@<0.5ex>[l]^{\Hom_A(U,-)}
 \ar@<1ex>[u]^{\Res}.
}
\]
\end{cor}

\begin{proof}
Since $U\in\add_R(T)$, Corollary~\ref{cor:asc+1} shows that
$U$ is Wakamatsu tilting over $A$. By Theorem~\ref{thm:end-extension}, the natural homomorphism
$\rho:S\to B$ is a Frobenius extension, and there is an isomorphism of
$R$-$B$-bimodules $U\cong T\otimes_S B$. %{eq:U-TB-BBM}

By assumption, $T$ has finite projective dimension both as a left
$R$-module and as a right $S$-module. Since $A$ is projective as a right $R$-module,
tensoring a finite projective resolution of ${}_RT$ with $A$ shows that
$\pd_A U<\infty$. Note that $\iota: S\to B$ is Frobenius, then $B$ is projective as a left $S$-module. Hence,
invoking the isomorphism of right $B$-modules $U\cong T\otimes_S B$, by
tensoring a finite projective resolution of $T_S$ with $B$ yields a
finite projective resolution of $U_B$. Therefore, $\pd_{B^{\op}}U<\infty$.
Consequently, Proposition~\ref{prop:BBM} applies to both
${}_RT_S$ and ${}_AU_B$, and we have the associated Brenner--Butler--Miyashita equivalences.

We next verify that extension and restriction of scalars preserve the
corresponding orthogonal classes.
Let $M\in\KT(T_S)$. Since both $\iota: R\to A$ and
$\rho: S\to B$ are Frobenius extensions, the relevant restriction and
extension functors are exact and preserve projectives. Therefore, standard isomorphism under
change of rings gives
\[ \Tor_i^B(U,B\otimes_S M)\cong\Tor_i^S(U, M)\cong
A\otimes_R\Tor_i^S(T, M)=0 \]
for all $i>0$. Thus $B\otimes_S M\in\KT(U_B)$.
Conversely, if $X \in\KT(U_B)$, then the isomorphism $U\cong T\otimes_S B$ %{eq:U-TB-BBM}
yields
\[ \Tor_i^S(T, X)\cong \Tor_i^B(T\otimes_SB,X)\cong\Tor_i^B(U,X)=0 \]
for all $i>0$. Hence $X\in\KT(T_S)$ as an $S$-module.

Now let $N\in\KE({}_RT)$. Since $U\in\add_R(T)$ and $\Ext_R^i(T,N)=0$ for all $i>0$,
by Lemma~\ref{lem:isoms} we have
\[\Ext_A^i(U,A\otimes_R N)\cong\Ext_R^i(U, N) =0.\]
Hence $A\otimes_R N\in\KE({}_AU)$.
Conversely, if $Y\in\KE({}_AU)$, then
$\Ext_R^i(T, Y)\cong \Ext_A^i(U, Y)=0$ for all $i>0$.
Therefore, as a restricted $R$-module, $Y\in\KE({}_RT)$.

The commutativity of the diagram follows from the same tensor--Hom
adjunction arguments as in Proposition~\ref{prop:add-compat}; we omit
the details.
\end{proof}

We conclude with two specific examples.

\begin{exm}\label{exm:RG}%4.4
Let $G$ be a finite group. For every ring $R$, since
${}_RRG_R\cong\bigoplus_{g\in G}{}_RR_R$, the extension of group ring
$R\rightarrow RG$ is split, centrally projective and Frobenius.
Hence, for every left $R$-module $T$,
$\End_R(T)\rightarrow \End_{RG}(RG\otimes_RT)$
is Frobenius. Moreover, Wakamatsu tilting is preserved and reflected:
$T$ is Wakamatsu tilting over $R$ if and only if
$RG\otimes_R T$ is Wakamatsu tilting over $RG$.
Whenever the hypotheses of finite projective dimension in
Corollary~\ref{cor:BBM-compat} hold, the corresponding
Brenner--Butler--Miyashita equivalences are compatible with induction
from $R$ to $RG$.
The same conclusions also hold for the truncated polynomial extension
$R\rightarrow R[x]/(x^n)$ ($n\geq 2$) when $x$ is central.
\end{exm}

Central projectivity guarantees the condition
$A\otimes_RT\in \add_R(T)$,
but the latter may also hold for Frobenius extensions that are not centrally
projective, as the following example shows.

\begin{exm}\label{exm:skew-group-end}%4.5
Let a finite group $G$ act on $R$ by automorphisms
$\{\alpha_g\}_{g\in G}$ and let $A=R*G$
be the skew group ring. Then $R\to A$ is a split Frobenius extension.
Moreover, ${}_RA_R\cong\bigoplus_{g\in G}{}_RR_{\alpha_g}$,
so the extension $R\to A$ need not be centrally projective.

For an $R$-module $T$, let ${}^{g}T$ denote the twisted module defined by
$r\cdot t=\alpha_{g^{-1}}(r)t$.
Then $A\otimes_RT\cong\bigoplus_{g\in G}{}^{g}T$ as left $R$-modules.
Hence, if ${}^{g}T\in\add_R(T)$, Theorem~\ref{thm:end-extension} yields a Frobenius extension
$\End_R(T) \rightarrow \End_A(A\otimes_RT)$.
In particular, this applies whenever $T$ is $G$-stable,
that is, ${}^{g}T\cong T$ for all $g\in G$. In this case
$A\otimes_RT\in\add_R(T)$, and since $R\to R*G$ is split, the ascent
and descent results show that $T$ is Wakamatsu tilting over $R$ if and only if
$A\otimes_R T$ is Wakamatsu tilting over $A$. Thus this example yields preservation and reflection of Wakamatsu
tilting modules even when the Frobenius extension is not centrally
projective.
\end{exm}

\vskip 10pt

\noindent {\bf Acknowledgements.}\quad This work was supported by the
Natural Science Foundation of Chongqing, China (No. CSTB2025NSCQ-GPX1014).

\bibliography{}

\vskip 10pt

{\footnotesize \noindent Wei Ren, Chunxia Zhang\\
 School of Mathematical Sciences, Chongqing Normal University, Chongqing 401331, PR China\\
 E-mail: {\tt wren$\symbol{64}$cqnu.edu.cn}, {\tt cxzhang$\symbol{64}$cqnu.edu.cn} }

\end{document}